\documentclass[runningheads,a4paper,envcountsame]{llncs}

\usepackage{amsfonts}
\usepackage{amssymb}
\usepackage{mathtools}
\usepackage[all,cmtip]{xy}
\usepackage[bookmarks=false]{hyperref} %bookmarks=false necessary for arxiv otherwise doesn't compile

\newcommand{\vir}[1]{``#1''}
\renewcommand{\phi}{\varphi}

\begin{document}

\title{Randomising Realisability\thanks{The authors would like to thank Rosalie Iemhoff and Jaap van Oosten for discussions about the material included in this paper.}}
\author{Merlin Carl\inst{1} \and Lorenzo Galeotti\inst{2} \and Robert Passmann\inst{3,4}}
\institute{Europa-Universit\"at Flensburg, 24943 Flensburg, Germany \and Amsterdam University College, Postbus 94160, 1090 GD Amsterdam, The Netherlands \and Institute for Logic, Language and Computation, Faculty of Science, University of Amsterdam, P.O. Box 94242, 1090 GE Amsterdam, The Netherlands \and St John's College, University of Cambridge, Cambridge CB2 1TP, England}

%
\maketitle              % typeset the header of the contribution
\begin{abstract}
We consider a randomised version of Kleene's realisability interpretation of intuitionistic arithmetic in which computability is replaced with randomised computability with positive probability. In particular, we show that (i) the set of randomly realisable statements is closed under intuitionistic first-order logic, but (ii) different from the set of realisable statements, that (iii) "realisability with probability 1" is the same as realisability and (iv) that the axioms of bounded Heyting's arithmetic are randomly realisable, but some instances of the full induction scheme fail to be randomly realisable.
\end{abstract}
\section{Introduction}
Have you met skeptical Steve? %Most mathematicians are quite skeptical, but skeptical Steve goes an extra mile: 
Being even more skeptical than most mathematicians, he only believes what he actually sees. To convince him that there is an $x$ such that $A$, you have to give him an example, together with evidence that $A$ holds for that example. To convince him that $A \rightarrow B$, you have to show him a \emph{method} for turning evidence of $A$ into evidence of $B$, and so on. Given that Steve is \vir{a man provided with paper, pencil, and rubber, and subject to strict discipline} \cite{IntelligentM1}, we can read \vir{method} as \vir{Turing program}, which leads us to Kleene's realisability interpretation of intuitionistic logic \cite{Kleene1945}.
%. This approach was taken by Kleene in developing his realisability interpretation of intuitionistic logic \cite{Kleene1945}.

Steve has a younger brother, pragmatical Per. Like Steve, Per is equipped with paper and pencil; however, he also has a coin on his desk, which he is allowed to throw from time to time while performing computations. %; this allows him to produce random strings. 
By his pragmatical nature, he does not require being successful at obtaining evidence for a given proposition $A$ every time he gives it a try; he is quite happy when it works with probability $(1-\frac{1}{10^{100}})$ or so, which makes it highly unlikely to ever fail in his lifetime. % (and certainly much less likely than making slips of pen while performing \vir{strict} computations). %If required, he can increase his success probability as much as he wants by making n tries instead of a single one.

Per wonders whether his pragmatism is more powerful than Steve's method. After all, he knows about Sacks's theorem \cite[Corollary 8.12.2]{downey2010algorithmic} that every function $f:\omega\rightarrow\omega$ that is computable using coin throws with positive probability is recursive. Can he find evidence for some claims where Steve fails? He also notices that turning such \vir{probabilistic evidence} for $A$ into \vir{probabilistic evidence} for $B$ is a job considerably different (and potentially harder) than turning evidence for $A$ into evidence for $B$. Could it be that there are propositions whose truth Steve can see, but Per cannot? Although Per is skeptical, e.g., of the law of the excluded middle just like Steve, he is quite fond of the deduction rules of intuitionistic logic; thus, he wonders whether the set of statements for which he can obtain his \vir{highly probably evidence} is closed under these.

Steve is unhappy with his brother's sloppiness. After all, even probability $(1-\frac{1}{10^{100}})$ leaves a small, albeit nonzero, chance of getting things wrong. He might consider changing his mind if that chance was brought down to $0$ by strengthening Steve's definition, demanding that the \vir{probabilistic evidence} works with probability $1$. However, he is only willing to give up absolute security if that leads to evidence for more statements. Thus, he asks whether \vir{probability $1$ evidence} is the same as \vir{evidence}.

These and other questions will be considered in this paper. To begin with, we will model Per's attitude formally, which gives us the concepts of \emph{$\mu$-realisability} and \emph{almost sure realisability}. We will then show the following:  There are statements that are $\mu$-realisable, but not realisable (Theorem \ref{Lemma:MurealRealREC}). The set of $\mu$-realisable statements are closed under deduction in intuitionistic predicate calculus (Theorem \ref{Theorem:Soundness}); in a certain sense to be specified below, the law or excluded middle fails for $\mu$-realisability (Lemma \ref{Lemma:LEM}). The axioms of Heyting arithmetic except for the induction schema are $\mu$-realised (Theorem \ref{Theorem Heyting Aritmetic}); and there are instances of the induction schema that are not $\mu$-realised (Theorem \ref{Theo:IndFail}). Almost sure realisability is the same as realisability (Theorem \ref{Theorem: F-realisability and realisability are the same}).

\section{Preliminaries}

Realisability is one of the most common semantic tools for the study of constructive theories and was introduced by Kleene in his seminal 1945 paper \cite{Kleene1945}. In this work, Kleene connected intuitionistic arithmetic---nowadays called \emph{Heyting arithmetic}---and recursive functions. The essential idea is that a statement is true if and only if there is a recursive function witnessing its truth. For more details on realisability, see also Troelstra's 344 \cite{Troelstra344}, and van Oosten's paper \cite{vanOosten2002} for an excellent historical survey of realisability. In particular, see \cite[Definition 3.2.2]{Troelstra344} for a definition of realisability in terms of recursive functions. In what follows, we denote this classical relation of realisability by `$\Vdash$'. 

As mentioned in the introduction, we want to give pragmatic Per the ability to throw coins while he tries to prove the truth of a statement. We will implement this coin throwing by allowing Per to access an infinite binary sequence. Therefore, we will make use of the Lebesgue measure on Cantor space $2^\omega$. For a full definition, see Kanamori's section on `Measure and Category' \cite[Chapter 0]{kanamori}. We denote the Lebesgue measure by $\mu$. Recall that a set $A$ is Lebesgue measurable if and only if there is a Borel set $B$ such that the symmetric difference of $A$ and $B$ is null. Given an element $u$ of Cantor space we will denote by $\mathrm{N}_{u\upharpoonright n}$ the basic clopen set $\{v\in 2^\omega \,;\, u\upharpoonright n \subset v\}$ where as usual $u\upharpoonright n$ is the prefix of $u$ of length $n$, and $u\upharpoonright n \subset v$ if $u\upharpoonright n$ is a prefix of $v$. We recall that given a binary sequence of length $n$, we have that $\mathrm{N}_{s}$ is measurable and  $\mu(\mathrm{N}_{s})=\frac{1}{2^{n}}$.

We fix a computable enumeration $(p_n)_{n\in \mathbb{N}}$ of programs. Moreover, given a program $p$ that uses an oracle and an element $u\in 2^\omega$ we will denote by $p^u$ the program $p$ where the oracle tape contains $u$ at the beginning of the computation. Moreover, given $n\in \mathbb{N}$ we will denote by $p(n)$ the program that for every oracle $u\in 2^\omega$ returns $p^u(n)$.  

A sentence in the language of arithmetic is said to be $\Delta_0$ if it does not contain unbounded quantifiers. We will say that a sentence is a \emph{pretty $\Sigma_1$} if it is $\Delta_0$ or of the form $Q_0Q_1\ldots Q_n\psi$ where $\psi$ is $\Delta_0$ and $Q_i$ is either an existential quantifier or a bounded universal quantifier for every $0\leq i\leq n$. Similarly, we will say that a sentence is a \emph{universal $\Pi_1$} if it is $\Delta_0$ or of the form $Q_0Q_1\ldots Q_n\psi$ where $\psi$ is $\Delta_0$ and $Q_i$ is  a universal quantifier for every $0\leq i\leq n$

%\begin{remark}
%Recall that the truth of $\Delta_0$ formula is decidable. Moreover, note that every pretty $\Sigma_1$ sentence is $\Sigma_1$ in the usual sense and its truth is therefore semidecidable.
%\end{remark}

Throughout this paper, we fix codings for formulas and programs. In order to simplify notation, we will use $\phi$ to refer to both the formula and its code, and similar for programs $p$.
We end this section with some lemmas on realisability of pretty $\Sigma_1$ and universal $\Pi_1$ formulas.

\begin{lemma}\label{Lem:CompClassicalTruth}
There is a program $p$ that for every pretty $\Sigma_1$ sentence $\varphi$ does the following:
If $\varphi$ is true then $p(\varphi)$ halts and outputs a realiser of $\varphi$ and otherwise, it diverges. 
\end{lemma}
\begin{proof}
First we define the program for $\Delta_0$ formulas by recursion.

\bigskip

(1) If $\varphi$ is atomic, $p$ first checks if $\varphi$ is true. If so then $p$ returns any natural number, otherwise, it loops.

\bigskip

(2) $\varphi\equiv \psi_0\land \psi_1$: the program $p$  checks whether $\psi_0$ and $\psi_1$ are true. If both computations are successful then $p$ returns the code of a program $q$ which returns $p(\psi_0)$ on input $0$ and $p(\psi_1)$ on input $1$.

\bigskip

(3) $\varphi\equiv \psi_0\lor \psi_1$ the program $p$ starts checking if at least one between $\psi_0$ and $\psi_1$ is true. If one of the two computations is successful then $p$ returns the code of a program $q$ which returns $p(\psi_i)$ on input $1$ and $i$ on input $0$ where $i$ is the smallest $i$ such that $\psi_i$ is true. Otherwise the program loops. 

\bigskip

(4) $\varphi\equiv \psi_0\rightarrow \psi_1$ then $p$ first checks whether $\psi$ is true if not $p$ returns $0$ otherwise returns a program that for every input returns $p(\psi_1)$.  

\bigskip

(5) $\varphi\equiv \exists{x<n}\psi$ then the program $p$ checks if there is $m<n$ such that $\psi(m)$ is true. If so then $p$ returns the code of a program $q$ which returns $p(\psi(m))$ on input $0$ and $m$ on input $1$ where $m$ is the smallest natural number such that $\psi(m)$ is true. Otherwise the program loops. 

\bigskip

(6) $\varphi\equiv \forall{x<n}\psi$ then $p$ checks in parallel the truth of all the instances of $\psi(m)$ for $m<n$. If all of them are true then $p$ returns the code of a program $q$ which for all $m\in \mathbb{N}$ returns $p(\psi(m))$. Otherwise the program loops.

\bigskip
Now we extend the definition of $p$ to pretty $\Sigma_1$ sentences. Assume that $\varphi$ is of the form $Q_1\ldots Q_n\psi$ where $\psi$ is $\Delta_0$ and $Q_i$ is either an existential quantifier or a bounded universal quantifier for every $0\leq i\leq n$. We define $p$ by recursion on $n$. Since the base case and the inductive step are essentially the same we will only show the latter. 

Let $n=m+1$ and $\varphi\equiv Q_0Q_1\ldots Q_n\psi$ where $\psi$ is $\Delta_0$. We assume that $f$ is already defined for $Q_1\ldots Q_n\psi$ and need to show that we can extend it to $\varphi$. We have two cases

\bigskip

(1) $Q_0$ is a bounded quantifier. Then we repeat what we did in part (5) and (6) of this proof. 

\bigskip

(2) $Q_0$ is an unbounded existential quantifier. Then the program $p$ starts an unbounded search to find an $i$ such that $\psi(i)$ is true. If it finds it then $p$ returns the code of a program $q$ which returns $p(\psi(i))$ on input $0$ and $i$ on input $1$ where $i$ is the smallest natural number such that $\psi(i)$ is true.
\end{proof}

\begin{lemma}\label{Lem:CompClassicalTruth2}
There is a program $p$ that for every universal $\Pi_1$ sentence $\varphi$ does the following:
If $\varphi$ is true then $p(\varphi)$ halts and outputs a realiser of $\varphi$ (we do not specify a behaviour otherwise).
\end{lemma}
\begin{proof}
Define $p$ as in the proof of Lemma \ref{Lem:CompClassicalTruth} for $\Delta_0$ formulas. Then for formulas of the type $\forall{x}\psi$ let $p(\psi)$ be the code of the program that for all $n$ runs $p(\psi(n))$. 
\end{proof}

\begin{lemma}\label{Lem:Sigma1Real}
A pretty $\Sigma_1$ sentence in the language of arithmetic is realised if and only if it is true. The same result holds for universal $\Pi_1$ sentences. 
\end{lemma}
\begin{proof}
The right-to-left direction follows from Lemma \ref{Lem:CompClassicalTruth}. The other direction is a straightforward induction on the complexity of $\phi$. The proof for universal $\Pi_1$ sentences is also an easy induction.
\end{proof}

\begin{corollary}\label{Cor:CompClassicalRealisability}
There is a program $p$ that for every pretty $\Sigma_1$ sentence $\varphi$ does the following:
If $\varphi$ is true then $p(\varphi)$ halts and outputs a realiser of $\varphi$ and otherwise, it diverges.
\end{corollary}

\section{Random Realisability}

In this section we will introduce the notion of $\mu$-realisability and prove the basic properties of this relation. As we mentioned before, we will modify classical realisability in order to use realisers that can access an element of Cantor space. Then we will say that a sentence is randomly realised if for non-null set of oracles in Cantor space the program does realise the sentence. Formally we define $\mu$-realisability as follows:

\begin{definition}[$\mu$-Realisability]\label{Def:MuReal}
We define two relations $\Vdash_\mathrm{O}$ and $\Vdash_{\mu}$ by mutual recursion. Let $u\in 2^\omega$, $p$ be a program that uses an oracle, and $\varphi$ be a sentence in the language of arithmetic. We define:
\begin{enumerate}
        \item $(p,u), \not\Vdash_O \bot$,
        \item $(p,u) \Vdash_\mathrm{O} n = m$ iff $n=m$,
        \item $(p,u) \Vdash_\mathrm{O} \phi \land \psi$ iff $(p^u(0),u) \Vdash_\mathrm{O} \phi $ and $(p^u(1),u) \Vdash_\mathrm{O} \psi$, 
         \item $(p,u) \Vdash_\mathrm{O} \phi \lor \psi$ iff we have $p^u(0)=0$ and $(p^u(1),u)\Vdash_{\mathrm{O}} \phi$ or $p^u(0)=1$ and $(p^u(1),u) \Vdash_{\mathrm{O}} \psi$, 
          \item $(p,u) \Vdash_\mathrm{O} \phi \rightarrow \psi$ iff  for all $s$ such that $s\Vdash_{\mu} \phi$, we have that $p^u(s) \Vdash_{\mu}\psi$,  
        \item $(p,u) \Vdash_\mathrm{O} \exists \mathbf{x} \phi$ iff $(p^u(0),u) \Vdash_\mathrm{O} \phi(p^u(1))$ ,
        \item $(p,u) \Vdash_\mathrm{O} \forall \mathbf{x} \phi$ iff for all $n\in \omega$ we have $(p^u(n),u) \Vdash_\mathrm{O} \phi(n)$, 
\end{enumerate}
For every program $p$ that uses an oracle and every sentence $\varphi$ in the language of arithmetic, we will denote by $C_{p,\varphi}$ the set:
$\{u\in 2^\omega\,;\, (p,u)\Vdash_{\mathrm{O}}\varphi\}.$
Let $\varphi$ be a sentence in the language of arithmetic, $r$ be a positive real number, and $p$ be a natural number. We define $p\Vdash_{\mu} \varphi\geq r$ as follows:
$p\Vdash_{\mu}\varphi \geq r\text{ iff } \mu(C_{p,\varphi})\geq r.$

In this case we will say that $p$ \emph{randomly realises (or $\mu$-realises) $\varphi$ with probability at least $r$}. We will say that $\phi$ is \emph{randomly realisable (or $\mu$-realisable) with probability at least $r$} if and only if there is $p$ such that $p\Vdash_{\mu}\varphi \geq r$. Moreover, we write $p\Vdash_{\mu}\varphi$ and say that $p$ \emph{randomly realises} (or $\mu$-realises) $\varphi$ if and only if $p\Vdash_{\mu}\varphi\geq r$ for some $r>0$. Finally, we will say that $\phi$ is \emph{randomly realisable (or $\mu$-realisable)} if $\sup\{\mu(C_{p,\varphi})\,;\, p\Vdash_{\mu} \varphi\}=1$. 
\end{definition}

%\begin{remark}
%Note that $\Vdash_\mu$ is well-defined due to the fact that $C_{p, \phi}$ is Lebesgue measurable for every program $p$ and sentence $\phi$. It can be shown by induction on $\phi$ that $C_{p, \phi}$ is Borel. This uses the fact that every program execution makes use of only finitely many digits of the oracle.
%\end{remark}

%\begin{remark}
%Definition \ref{Def:MuReal} is quite involved. 
%It is %therefore 
%natural to ask why one cannot define $\Vdash_\mu$ in a simpler way. 
Why is it not possible to give a simpler definition of $\Vdash_\mu$? A natural attempt would be the following:
$p\Vdash_\mu \varphi\geq r \Leftrightarrow \mu(\{ u\,;\, (p,u)\Vdash_{\mathrm{Or}} \varphi\})\geq r$
where $\Vdash_{\mathrm{Or}}$ denotes oracle-realisability, obtained by replacing computability with computability relative to a fixed oracle in Kleene realisability. %new: 28.01.2021
Unfortunately, %this definition turns out not to be very well behaved. In particular, 
it turns out that this relation is not closed under modus ponens and the $\forall$-$\mathrm{GEN}$ rule of predicate logic. Another natural approach is the one of \S\,\ref{Sec:BigReal}.
%\end{remark}
We start our study of $\mu$-realisability by showing that the set of $\mu$-realised sentences of arithmetic is consistent.

\begin{lemma}
Let $\varphi$ a sentence in the language of arithmetic. Then $\varphi$ is $\mu$-realised iff $\lnot \varphi$ is not $\mu$-realised.
\end{lemma}
\begin{proof}
Assume that both $p\Vdash_{\mu}\varphi$ and $q\Vdash_{\mu}\lnot \varphi$. Then for all $u\in C_{q,\lnot\phi}$ we have that $q^u(p)$ would be a realiser of $\bot$. But this is a contradiction.  
\end{proof}

The following lemma has a crucial role in the theory of $\mu$-realisability.

\begin{lemma}[Push Up Lemma]\label{Lemma:PushUP}
Let $\varphi$ be a sentence in the language of first order arithmetic and $0<r \leq r'<1$ be positive real numbers. Then $\varphi$ is randomly realisable with probability at least $r$ if and only if $\varphi$ is randomly realisable with probability at least $r'$.
\end{lemma}
\begin{proof}
The right-to-left direction is trivial. For the left-to-right direction, 
let $\varphi$ be randomly realisable with probability at least $r$. We will show that  $\varphi$ is randomly  with probability at least $r'$. Let $p$ be a program such that $\mu(C_{p,\varphi})\geq r>0$. By the Lebesgue Density Theorem \cite[Exercise 17.9]{Kechris} there are $u\in 2^\omega$ and $n\in \omega$ such that $\frac{\mu(C_{p,\varphi} \cap \mathrm{N}_{u\upharpoonright n})}{\mu(\mathrm{N}_{u\upharpoonright n})}>r'$. Now, let $p'$ be the program that given an oracle runs $p$ with oracle $(u\upharpoonright n)\circ u$. Note that $\mu(C_{p',\varphi})=\frac{\mu(C_{p,\varphi} \cap \mathrm{N}_{u\upharpoonright n})}{\mu(\mathrm{N}_{u\upharpoonright n})}>r'$. Finally, it follows trivially by the definition that $p'$ randomly realisable with probability at least $r'$ as desired.
\end{proof}

By the Push Up Lemma, we can simplify our definition of $\mu$-realisability. 

\begin{corollary}\label{Lem:Mu-real2Def}
A sentence $\varphi$ in the language of arithmetic is $\mu$-realisable if and only if there are $0<r\in \mathbb{R}$ and $p$ such that that $\mu$-realises $\varphi$ with probability at least $r$.  
\end{corollary}

We conclude this section by proving some basic interactions between $\mu$-realisability and the logical operators. 
%Note that if $p$ is a program and $\varphi$ is a sentence, then the set $C_{p,\varphi}$ is of low complexity in the following sense: for all programs $p$ and sentences $\varphi$ the set $C_{p,\varphi}$ is Borel, and therefore measurable.

\begin{lemma}\label{Lem:MasurabilityOfOracleSets}
For all programs $p$ and sentences $\varphi$ the set $C_{p,\varphi}$ is Borel. In particular $C_{p,\varphi}$ is measurable.
\end{lemma}
\begin{proof}
The proof is an induction on the complexity of $\varphi$. All the cases except implication follow directly from the closure properties of the pointclass of Borel sets, see, e.g., \cite[Theorem 1C.2]{moschovakis}. Let us just prove the implication case. Let $\varphi\equiv \psi_0\rightarrow \psi_1$ and $p$ be a program. For every program $s$ let $A_s$ be $2^\omega$ if $s\not \Vdash_{\mu} \psi_0$ and $C_{p(s),\psi_1}$, otherwise. Then $C_{p,\varphi}=\bigcap_{s\in \mathbb{N}}A_s$. By inductive hypothesis we have that $A_s$ is Borel for every $s$ so $C_{p,\varphi}$ is a countable intersection of Borel sets, which is Borel.   
\end{proof}

\begin{corollary}\label{Cor:MURealInductive}
Let $\psi_0$ and $\psi_1$ be sentences and let $\phi$ be a formula. Then for every $p$ the following hold:
\begin{enumerate}
\item\label{Cor:MURealInductive1}  $p\Vdash_{\mu}\psi_0\land \psi_1$ if and only if there are $s$ and $q$ such that $s\Vdash_{\mu}\psi_0$ and $q\Vdash_{\mu}\psi_1$.
\item\label{Cor:MURealInductive2}  $p\Vdash_{\mu}\psi_0\lor \psi_1$ if and only if there is $q$ such that $q\Vdash_{\mu}\psi_0$ or $q\Vdash_{\mu}\psi_1$. 
\item\label{Cor:MURealInductive3}  If $p\Vdash_{\mu}\psi_0\rightarrow \psi_1$ then $p(s)\Vdash_{\mu} \psi_1$ for all $s$ such that $s\Vdash_{\mu} \psi_0$.
\item\label{Cor:MURealInductive4}  If $p\Vdash_{\mu}\exists{x}\phi$ then there there is $n\in \mathbb{N}$ such that $p(0)\Vdash_{\mu}\phi(n)$.
\item\label{Cor:MURealInductive5}  If $p\Vdash_{\mu}\forall{x}\phi$ then for all $n\in \mathbb{N}$ we have $p(n)\Vdash_{\mu}\phi(n)$.
\end{enumerate}

\end{corollary}
\begin{proof}
Note that an case-by-case proof shows that for every $n\in \mathbb{N}$, $u\in 2^\omega$, and formula $\phi$ we have that:

$$(p(n),u)\Vdash_{\mathrm{O}} \phi \text{ iff } (p^{u}(n),u)\Vdash_{\mathrm{O}}\phi.$$

(\ref{Cor:MURealInductive1}) First assume that $p\Vdash_{\mu}\psi_0\land \psi_1\geq r$. Then we have that $\mu(C_{p,\psi_0\land \psi_1})\geq r$ and for all $u\in C_{p,\psi_0\land \psi_1}$ we have that $(p,u)\Vdash_{\mathrm{O}}\psi_0\land \psi_1$. But then for all $u\in C_{p,\psi_0\land \psi_1}$ we have that $(p^u(0),u)\Vdash_{\mathrm{O}}\psi_0$ and $(p^u(0),u)\Vdash_{\mathrm{O}}\psi_1$. So for every $u\in C_{p,\psi_0\land \psi_1}$ we have that $(p(0),u)\Vdash_{\mathrm{O}}\psi_0$ and $(p(1),u)\Vdash_{\mathrm{O}}\psi_1$ but then $p(0)\Vdash_{\mu}\psi_0\geq r$ and $p(1)\Vdash_{\mu}\psi_1\geq r$. Let $s=p(0)$ and $q$ be $p(1)$.

Now assume that $q\Vdash_{\mu}\psi_0$ and $s\Vdash_{\mu}\psi_1$. Therefore for all $u\in C_{q,\psi_0}$ and $v\in C_{s,\psi_1}$ we have that $(q,u)\Vdash_{\mathrm{O}}\psi_0$ and $(s,u)\Vdash_{\mathrm{O}}\psi_1$. By Lemma \ref{Lemma:PushUP} we can assume that $s$ and $q$ are such that  $\mu(C_{q,\psi_0}\cap C_{s,\psi_1})>r$. But then if we let $p$ be the program that returns $q$ on input $0$ and $s$ on input $1$, we have that for all $u\in C_{q,\psi_0}\cap C_{s,\psi_1}$ we have that  $(p(0),u)\Vdash_{\mathrm{O}}\psi_0$ and $(p(1),u)\Vdash_{\mathrm{O}}\psi_1$. Therefore $p\Vdash_{\mu}\psi_0\land \psi_1$ as desired.

(\ref{Cor:MURealInductive2}) First assume that $p\Vdash_{\mu}\psi_0\lor \psi_1$. Then we have that $\mu(C_{p,\psi_0\lor \psi_1})> 0$ and for all $u\in C_{p,\psi_0\lor \psi_1}$ we have that $(p,u)\Vdash_{\mathrm{O}}\psi_0\lor \psi_1$. But then for all $u\in C_{p,\psi_0\lor \psi_1}$ we have that $(p^u(1),u)\Vdash_{\mathrm{O}}\psi_0$ or $(p^u(1),u)\Vdash_{\mathrm{O}}\psi_1$. Moreover, note that by Lemma \ref{Lem:MasurabilityOfOracleSets} we have that both $C_{p(1),\psi_0}$ and $C_{p(1),\psi_1}$ are measurable sets and since $ C_{p,\psi_0\lor \psi_1} \subseteq C_{p(1),\psi_0}\cup C_{p(1),\psi_1}$ at least one of them is not null. Without loss of generality assume that $\mu(C_{p(1),\psi_0})>0$. So, $(p(1),u)\Vdash_{\mathrm{O}}\psi_0$ for all $u\in C_{p(1),\psi_0}$. Let $q$ be the program that for every oracle returns $p^{u}(1)$ on input $1$ and $0$ on input $0$. Then trivially $q\Vdash_{\mu} \psi_0$ as desired.

Now assume that $q\Vdash_{\mu}\psi_0$, the same proof works in the case in which $q\Vdash_{\mu}\psi_0$. Therefore for all $u\in C_{q,\psi_0}$ we have that $(q,u)\Vdash_{\mathrm{O}}\psi_0$. Then if we let $p$ be the program that returns $q$ on input $1$ and $0$ on input $0$, we have that for all $u\in C_{q,\psi_0}$ $(p(1),u)\Vdash_{\mathrm{O}}\psi_0$. Therefore $p\Vdash_{\mu}\psi_0\lor \psi_1$ as desired.

(\ref{Cor:MURealInductive3}) Assume that $p\Vdash_{\mu}\psi_0\rightarrow \psi_1$. Then we have that $\mu(C_{p,\psi_0\rightarrow \psi_1})> 0$ and for all $u\in C_{p,\psi_0\rightarrow \psi_1}$ we have that $(p,u)\Vdash_{\mathrm{O}}\psi_0\rightarrow \psi_1$. But then for all $u\in C_{p,\psi_0\rightarrow \psi_1}$ and for every $s$ such that $s\Vdash_{\mu}\psi_0$ we have $p^u(s)\Vdash_{\mu}\psi_1$. But for all $s$ such that $s\Vdash_{\mu}\psi_0$ we have that $(p^u(s),u)\Vdash_{\mathrm{O}}\psi_1$ and therefore $(p(s),u)\Vdash_{\mathrm{O}}\psi_1$. So, $p(s)\Vdash_{\mu}\psi_1$ for all $s$ such that $s\Vdash_{\mu}\psi_0$ as desired.

(\ref{Cor:MURealInductive4}) Assume that $p\Vdash_{\mu}\exists x\psi$. Then we have that $\mu(C_{p,\exists x\psi})> 0$ and for all $u\in C_{p,\psi_0\lor \psi_1}$ we have that $(p,u)\Vdash_{\mathrm{O}}\exists x\psi$. But then for all $u\in C_{p,\exists x\psi}$ we have that $(p^u(0),u)\Vdash_{\mathrm{O}}\psi(p^u(1))$. Moreover, note that by Lemma \ref{Lem:MasurabilityOfOracleSets} we have that for every $n\in \mathbb{N}$ the set $C_{p(0),\psi(n)}$ is measurable and since $C_{p,\exists x\psi} \subseteq  \bigcup_{n\in \mathbb{N}}C_{p(0),\psi(n)}$ there is $n\in \mathbb{N}$ such that $\mu(C_{p(0),\psi(n)})>0$. So, $(p(0),u)\Vdash_{\mathrm{O}}\psi(n)$ for all $u\in C_{p(0),\psi(n)}$ and therefore $p(0)\Vdash_{\mu}\psi(n)$ as desired.

(\ref{Cor:MURealInductive4}) Assume that $p\Vdash_{\mu}\forall x\psi$. Then we have that $\mu(C_{p,\forall x\psi})> 0$ and for all $u\in C_{p,\forall x\psi}$ we have that $(p,u)\Vdash_{\mathrm{O}}\forall x\psi$. But then for all $u\in C_{p,\exists x\psi}$ and every $n\in \mathbb{N}$ we have that $(p^u(n),u)\Vdash_{\mathrm{O}}\psi(n)$. So, $(p(n),u)\Vdash_{\mathrm{O}}\psi(n)$ and $p(n)\Vdash_{\mu}\psi(n)$ for all $n\in \mathbb{N}$. 
\end{proof}

\begin{corollary}\label{Cor:MURealInductiveProb1}
Let $\psi_0$ and $\psi_1$ be sentences and let $\phi$ be a formula. Then for every $p$ the following hold:
\begin{enumerate}
\item\ref{Cor:MURealInductive1}  $p\Vdash_{\mu}\psi_0\land \psi_1\geq 1$ if and only if $p(0)\Vdash_{\mu}\psi_0\geq 1$ and $p(1)\Vdash_{\mu}\psi_1\geq 1$.
\item If $p(1)\Vdash_{\mu}\psi_0\geq 1$ or $p(1)\Vdash_{\mu}\psi_1\geq 1$ then $p\Vdash_{\mu}\psi_0\lor \psi_1\geq 1$.
\item If $p(s)\Vdash_{\mu} \psi_1\geq 1$ for all $s$ such that $s\Vdash_{\mu} \psi_0$ then $p\Vdash_{\mu}\psi_0\rightarrow \psi_1\geq 1$.
\item If there there is $n\in \mathbb{N}$ such that $p(n)\Vdash_{\mu}\phi(n)\geq 1$ then $p\Vdash_{\mu}\exists{x}\phi\geq 1$.
\item If all $n\in \mathbb{N}$ we have $p(n)\Vdash_{\mu}\phi(n)\geq 1$ then $p\Vdash_{\mu}\forall{x}\phi\geq 1$.
\end{enumerate}
\end{corollary}
\begin{proof}
    The proof is an easy modification of the proof of Corollary \ref{Cor:MURealInductive}.
\end{proof}

\section{Classical Realisability and Random Realisability}
In this section, we will study the relationship between classical and random realisability. In particular we will show that the two notion do not coincide.

We start by proving that classical realisability and $\mu$-realisability agree on pretty $\Sigma_1$ sentences and that therefore $\mu$-realisability for pretty $\Sigma_1$ sentences coincides with truth.  

\begin{theorem}\label{Lem:RecClass}
Let $\phi$ be a pretty $\Sigma_1$ sentence in the language of arithmetic. Then, there are two computable functions $P_{\mu}$ and $P^{-1}_{\mu}$ such that for every $p$, (i) $p \Vdash \varphi$ implies $P_{\mu}(p,\varphi) \Vdash_{\mu}\varphi\geq 1$, and (ii) $p\Vdash_{\mu} \varphi$ implies $P_{\mu}^{-1}(p,\varphi)\Vdash\varphi$.
% \begin{align*}
% \text{For every $p$ we have }& p\Vdash \varphi \Rightarrow P_{\mu}(p,\varphi)\Vdash_{\mu}\varphi\geq 1\\
% \text{For every $p$ we have }& p\Vdash_{\mu} \varphi \Rightarrow P_{\mu}^{-1}(p,\varphi)\Vdash\varphi.
% \end{align*}
Therefore a pretty $\Sigma_1$ formula is true if and only if it is $\mu$-realised. The same result holds for universal $\Pi_1$ sentences.
\end{theorem}
\begin{proof}
We define $P_{\mu}$ and $P_{\mu}^{-1}$ by recursion on $\varphi$ and will prove that they have the desired properties. We first define $P_{\mu}$ and $P_{\mu}^{-1}$ on $\Delta_0$ formulas.
\bigskip

(1) $\varphi$ is atomic: in this case realisability and $\mu$-realisability have the same realisers. So we can just let $P_{\mu}$ and $P_{\mu}^{-1}$ be the identity on atomic formulas. 

\bigskip

(2) If $\varphi\equiv \psi_0\land \psi_1$ where $\psi_0$ and $\psi_1$ are $\Delta_0$.  

First assume that $p\Vdash_{\mu}\psi_0\land \psi_1$. Then, $\mu(C_{p,\phi})>0$ and for every $u\in C_{p,\phi}$ we have $(p^u(0),u)\Vdash_{\mathrm{O}} \psi_0$ and $(p^{u}(1),u)\Vdash_{\mathrm{O}} \psi_1$. Let $p(i)$ be the program that for every oracle $u$ just returns $p^u(i)$ for $i\in \{0,1\}$. Then for $i\in \{0,1\}$ we have $p(i)\Vdash_{\mu}\psi_i$. Let $P_{\mu}^{-1}(p,\varphi)$ be the program that given input $i$ computes $g(p(i),\psi_i)$. 
By inductive hypothesis we have that $P_{\mu}^{-1}(p(i),\psi_i)\Vdash \psi_i$ for every $i\in \{0,1\}$ and therefore $P_{\mu}^{-1}(p,\varphi)$ realises $\varphi$ as desired.

On the other hand let $p\Vdash\psi_0\land \psi_1$. Then $p(i)\Vdash \psi_i$ for every $i\in \{0,1\}$. Let $P_{\mu}(p,\varphi)$ be the program that ignores the oracle and for every $i\in \{0,1\}$ returns $P_{\mu}(p(i),\psi_i)$. By inductive hypothesis we have that $P_{\mu}(p(i),\psi_i)\Vdash_\mu \psi_i\geq 1$. Note that $\mu(C_{P_{\mu}(p(0),\psi_0),\psi_0}\cap C_{P_{\mu}(p(1),\psi_1),\psi_1})=1$ and that for all $u\in C_{P_{\mu}(p(0),\psi_0),\psi_0}\cap C_{P_{\mu}(p(1),\psi_1),\psi_1}$ we have that $(P_{\mu}(p(0),\psi_0),u)\Vdash_{\mathrm{O}}\psi_0$ and $(P_{\mu}(p(1),\psi_1),u)\Vdash_{\mathrm{O}}\psi_1$. But then, since $P_{\mu}(p,\varphi)^u(i)=P_{\mu}(p(i),\psi_i)$ for every oracle $u$ and every $i\in \{0,1\}$, we have that $P_{\mu}(p,\varphi)\Vdash_{\mu}\psi_0\land \psi_1\geq 1$.

\bigskip

(3) If $\varphi\equiv \psi_0\lor \psi_1$ where $\psi_0$ and $\psi_1$ are $\Delta_0$. 

First assume that $p\Vdash_{\mu}\psi_0\lor \psi_1$. Then, $\mu(C_{p,\phi})>0$ and for every $u\in C_{p,\phi}$ we have that $(p^u(1),u)\Vdash_{\mathrm{O}}\psi_{p^u(0)}$. Let $p(1)$ be the program that for every oracle $u$ returns $p^u(1)$. Then there is $i\in \{0,1\}$ such that $p(1)\Vdash_{\mu}\psi_i$ by the proof of Corollary \ref{Cor:MURealInductive}. Let $P_{\mu}^{-1}(p,\varphi)$ be the program that for does the following:
starts by running in parallel two instances of the program of Corollary \ref{Cor:CompClassicalRealisability}, one with input $\psi_0$ and one with input $\psi_1$. By inductive hypothesis note that at least one of the two instances will halt. Let $i\in \{0,1\}$ be such that the $\psi_i$ instance halted first. Then, if the input is $0$, the program halts with output $P_{\mu}^{-1}(p(0),\psi_i)$, and if the input is $1$ the program returns $i$. 

Now let $p\Vdash\psi_0\lor \psi_1$ where $\psi_0$ and $\psi_1$ are $\Delta_0$. Then $p(1)\Vdash \psi_{p(0)}$. Let $P_{\mu}(p,\varphi)$ be the program that if the input is $0$ halts with output $P_{\mu}(p(1),\psi_{p(0)})$, and if the input is $1$ the program returns $p(0)$. 

By inductive hypothesis we have that $P_{\mu}(p(1),\psi_{p(0)})\Vdash_\mu \psi_{p(0)}\geq 1$. But then, by the proof of Corollary \ref{Cor:MURealInductiveProb1} 
%for all $u \in C_{P_{\mu}(p(0),\psi_i),\psi_{i}}$ we have $(P_{\mu}(p(0),\psi_{i}),u) \Vdash_{\mathrm{O}}\psi_{i}$.
since for every $u$, $P_{\mu}(p,\varphi)^u(1)=P_{\mu}(p(1),\psi_{p(0)})$, and $P_{\mu}(p,\varphi)(0)=p(0)$,  
we have $P_{\mu}(p,\varphi)\Vdash_\mu\psi_0\lor \psi_1\geq 1$ as desired.

\bigskip

(4) If $\varphi\equiv \psi_0\rightarrow \psi_1$ where $\psi_0$ and $\psi_1$ are $\Delta_0$.  First assume that $p\Vdash_{\mu}\psi_0\rightarrow \psi_1$. Let $P_{\mu}^{-1}(p,\phi)$ be the program that does the following: for every input returns the code of the instance of the program in Corollary \ref{Cor:CompClassicalRealisability} with input $\psi_1$. Note that, if $\psi_0$ is realisable, then by inductive hypothesis is $\mu$-realisable, by assumptions $\psi_1$ is $\mu$-realisable and by inductive hypothesis $\psi_1$ is realisable. In this case  for every $s$ we have that $P_{\mu}^{-1}(p,\phi)(s)\Vdash \psi_1$ and therefore $P_{\mu}^{-1}(p,\phi)$ is a realiser of $\psi_1$. On the other hand, if $\psi_1$ is not realisable, then any natural number realises $\phi$, so $P_{\mu}^{-1}(p,\phi)$ is again a realiser of $\phi$.

Now let $p\Vdash\psi_0\rightarrow \psi_1$ where $\psi_0$ and $\psi_1$ are $\Delta_0$. Then for every realiser $s$ of $\psi_0$ we have that $p(s)\Vdash \psi_{1}$. Let $P_{\mu}(p, \phi)$ be the code of the program that for every input $s$ and every oracle returns $P_{\mu}(p(P_{\mu}^{-1}(s,\psi_0)),\psi_1)$. By inductive hypothesis if $s$ is a $\mu$-realiser of $\psi_0$, so $P_{\mu}^{-1}(s,\psi_0)$ is a realiser of $s$. By assumption $p(P_{\mu}^{-1}(s,\psi_0))$ is a realiser of $\psi_1$, and again by inductive hypothesis we have that $P_{\mu}(p(P_{\mu}^{-1}(s,\psi_0),\psi_1)\Vdash_{\mu}\psi_1\geq 1$. But then by Corollary \ref{Cor:MURealInductiveProb1} we have that $P_{\mu}(p,\phi)\Vdash_{\mu}\phi\geq 1$ as desired.

\bigskip

(5) we omit the bounded quantifier cases because they are analogous to the conjunction and disjunction cases.

\bigskip

Now we extend the definition to pretty $\Sigma_1$ formulas.

\bigskip

(6) If $\varphi\equiv \exists{x} \psi$ where $\psi$ is pretty $\Sigma_1$. First assume that $p\Vdash_{\mu}\exists{x} \psi$. Then, by Corollary \ref{Cor:MURealInductive} there must be $n\in \mathbb{N}$ such that $p(0)\Vdash_{\mu}\psi(n)$ therefore, by inductive hypothesis, $\psi(n)$ is realised. Let $P_{\mu}^{-1}(p,\varphi)$ be the program that does the following: run in parallel all the instances of the program of Corollary \ref{Cor:CompClassicalRealisability} with input $\psi(n)$ with $n\in \mathbb{N}$. By inductive hypothesis note that one of these instances must halt. Let $i\in \mathbb{N}$ be the least such that the $\psi(i)$ instance halts. Then, if the input is $0$, the program returns $P_{\mu}^{-1}(p(0),\psi(i))$ and if it is $1$, the program returns $i$. 

Note that, by inductive hypothesis, the program halts and returns a realiser of $\phi$, as desired.

Now assume that $p\Vdash\exists{x}\psi$. Then $p(0)\Vdash \psi(p(1))$. Let $f(p,\varphi)$ be the program that returns $P_{\mu}(p(0),\psi(p(1)))$ if the input is $0$ and $p(1)$ if the input is $1$. By inductive hypothesis $P_{\mu}(p(0),\psi(p(1)))\Vdash_{\mu}\psi(p(1))\geq 1$. But then by Corollary \ref{Cor:MURealInductiveProb1} since for all $u$ we have $P_{\mu}(p,\varphi)^u(0)=f(p(0),\psi(p(1)))$ and $P_{\mu}(p,\varphi)^u(0)=p(1)$, we have that $P_{\mu}(p,\varphi)\Vdash_{\mu}\exists{x}\psi\geq 1$ as desired.

\bigskip

(7) If $\varphi\equiv \forall{x<n} \psi$, where $\psi$ is pretty $\Sigma_1$. First assume that $p\Vdash_{\mu}\forall{x<n} \psi$. Then, by Corollary \ref{Cor:MURealInductive}, for every natural number $m<n$ and for every program $s$ we have that $p(m)\Vdash_{\mu}\psi(m)$ and by inductive hypothesis $\psi(m)$ is realised. Let $P_{\mu}^{-1}(p,\varphi)$ be the program that for every $m$ returns a program that for every input if $m<n$ returns $P_{\mu}^{-1}(p(m), \psi(m))$ and returns $0$ otherwise. For all $m<n$, by inductive hypothesis we have that $P_{\mu}^{-1}(p(m), \psi(m)) \Vdash \psi(m)$ and therefore $P_{\mu}^{-1}(p,\varphi)$ is a realiser of $\phi$ as desired.

Now assume that $p\Vdash\forall{x<n}\psi$. Then for all $m<n$ and $s$ we have that $p(m)(s)\Vdash \psi(m)$. Let $P_{\mu}(p,\varphi)$ be the program that ignores the oracle and for every $m$ returns a program that given $q$ as input, if $m<n$ then returns $P_{\mu}(p(m)(q), \psi_1)$ otherwise returns $0$. Now note that for every $m<n$ and for every program $q$ we have that $P_{\mu}(p(m)(q), \psi_1)\Vdash_\mu \psi(m)\geq 1$. But then for every $m<n$, every $q$, and every $u\in 2^\omega$ we have that $(P_{\mu}(p,\varphi)^u(m))^u(q)\Vdash_{\mu} \psi\geq 1$, and therefore by Corollary \ref{Cor:MURealInductiveProb1} we have $P_{\mu}(p,\varphi)\Vdash_{\mu}\phi\geq 1$ as desired.

\bigskip

Finally we extend the $\Delta_0$ case to universal $\Pi_1$ formulas. 

\bigskip

(8) If $\varphi\equiv \forall{x} \psi$ where $\psi$ is universal $\Pi_1$.
First assume that $p\Vdash_{\mu}\forall{x} \psi$. Let $P_{\mu}^{-1}(p,\varphi)$ be the program that for all $n$ runs $P_{\mu}(P_{\mu}^{-1}(p(n),\psi(n)))$. By Corollary \ref{Cor:MURealInductive} and the inductive hypothesis $P_{\mu}^{-1}(p(n),\psi(n))$ is a realiser of $\psi(n)$. Therefore $P_{\mu}^{-1}(p,\varphi)$ is a realiser of $\forall {x} \psi$ as desired.

Now assume that $p\Vdash\forall{x}\psi$. Let $P_{\mu}(p,\phi)$ be the program that for every $n$ and for every oracle returns $P_{\mu}(p(n),\psi(n))$. Once more by inductive hypothesis for all $n$ and all $\mu(C_{P_{\mu}(p(n),\psi(n)),\psi(n)})=1$ but then $\mu(\bigcap_{n\in \mathbb{N}}C_{P_{\mu}(p(n),\psi(n)),\psi(n)})=1$ and $P_{\mu}(p,\varphi)\Vdash_\mu \varphi\geq 1$ as desired.

The second part of the statement follows from Lemma \ref{Lem:Sigma1Real}.
\end{proof}

\begin{corollary}\label{Lemma:NEG}
Let $\varphi$ be any false pretty $\Sigma_1$ sentence in the language of arithmetic. Then $(p,u)\Vdash_{\mathrm{O}}\varphi\rightarrow \bot$ and $p\Vdash_{\mu}\varphi\rightarrow \bot\geq 1$ for every $p$ and $u$. The same holds for universal $\Pi_1$ formulas.
\end{corollary}
\begin{proof}
By Theorem \ref{Lem:RecClass} every $\mu$-realisable pretty $\Sigma_1$ (universal $\Pi_1$) sentence $\phi$ is true. Therefore $\phi$ cannot be $\mu$-realised and every program is going to $\mu$-realise $\phi\rightarrow \bot$, which means that for all $p$ and for all $u$ we have $p\Vdash_{\mu}\phi\rightarrow \bot$ and  $(p,u)\Vdash_{\mathrm{O}}\phi\rightarrow \bot$ as desired.
\end{proof}

We are now ready to prove the main result of this section, namely that $\mu$-realisability and classical realisability do no coincide. This result is surprising given that by Sacks's theorem \cite[Corollary 8.12.2]{downey2010algorithmic} functions that are computable with a non-null set of oracles are computable by a classical Turing machine.  

\begin{theorem}\label{Lemma:MurealRealREC}
There is a sentence $\varphi$ in the language of arithmetic that is randomly realisable but not realisable.
\end{theorem}
\begin{proof}
Let $\varphi$ be the sentence \vir{For all $k$ there is $n$ such that for all $\ell$ the execution of $p_k(k)$ does not stop in at most $\ell$ steps or $p_k(k)\neq n$} and let $\psi(k)$ be the sentence \vir{There is $n$ such that for all $\ell$ the execution of $p_k(k)$ does not stop in at most $\ell$ steps or $p_k(k)\neq n$}.

A classical realiser for $\varphi$ would be a program that computes a total function that, for every code $k$ of a program, returns a natural number which is not the output of $p_k(k)$. By  diagonalization, such a program cannot exists: If $p_k$ was such a program, then it would follow that for every $n\in \omega$ we have that $p_k(k)=n\Leftrightarrow p_k(k)\neq n$.

Now we want to show that $\varphi$ is randomly realisable. 

Fix any realiser $s$. Let $p$ be the program that given an oracle $u\in 2^\omega$, a natural number $k$, and $i\in \{0,1\}$ does the following\footnote{Here, we do not distinguish between the finite sequence $u{\upharpoonright}(k+1)$ and the natural number coding it.}: Let $p^u(k)(i)=u{\upharpoonright}(k+1)$ if $i=1$ and $p^u(k)(i)=p'$ if $i=0$ where $p'$ is the program that ignores the oracle and does the following: 

On input $\ell$, $p^{\prime}$ checks whether $p_k(k)$ stops in $\ell$ steps. If not, then $p'(\ell)$ is the code of a program that returns $0$ on input $0$ and $s$ on. input $1$. Otherwise  $p'(\ell)$ is the code of a program that returns $1$ on input $0$ and on input $1$ looks for an $\mu$-realiser of the $\Delta_0$ formula expressing the fact that \vir{$p_k(k)\neq u{\upharpoonright}(k+1)$} by running the algorithms in Lemma \ref{Lem:Sigma1Real} and Theorem \ref{Lem:RecClass}.

Now, for every $k\in \omega$ and $u\in 2^\omega$ we have two cases:

$p_k(k)$ does not halt: then we have that $p^u(k)(1)=u{\upharpoonright}(k+1)$ and $p^u(k)(0)=p'$. Since $p_k(k)$ does not halt, we have $p'(\ell)(0)=1$ and $p'(\ell)(1)=s$ for every $\ell$. Moreover, by Corollary \ref{Lemma:NEG}  $(s,u)\Vdash_{\mathrm{O}} \text{\vir{$p_k(k)$ does not halt in $\ell$ steps}}$ and therefore $(p^u(k),u) \Vdash_{\mathrm{O}}\psi(k)$.

$p_k(k)$ halts: then we have that $p^u(k)(1)=u{\upharpoonright}(k+1)$ and $p^u(k)(0)=p'$. Let $\ell$ be such that $p_k(k)$ halts in at most $\ell$ steps. Then, $p'(\ell)(0)=0$. Moreover, note that if the output of $p_k(k)$ is not the same as the first $k$ bits of the oracle then $(p'(\ell)^u(1),u)\Vdash _{\mathrm{O}} u{\upharpoonright}(k+1) \neq p_k(k)$. 

We only need to show that $\mu(C_{p,\varphi})>0$. To see this, it is enough to note that the set of $u$ such that $p_k(k)\neq u{\upharpoonright}(k+1)$ has measure $\geq 1-\frac{1}{2^{(k+1)}}$. Therefore, $\mu(C_{p,\varphi})=\prod_{k\in \mathbb{N}}(1-\frac{1}{2^{(k+1)}})>0$ as desired.
\end{proof}

\begin{corollary}\label{Cor:Sentence realisable but not randomly}
There is a sentence in the language of arithmetic which is realisable but not randomly realisable.
\end{corollary}
\begin{proof}[Corollary \ref{Cor:Sentence realisable but not randomly}]
It is enough to consider the sentence $\varphi\rightarrow \bot$ where $\varphi$ is the sentence in the proof of Theorem \ref{Lemma:MurealRealREC}. The sentence is trivially realised since $\varphi$ is not realised. Moreover, the sentence is not $\mu$-realised since $\varphi$ is $\mu$-realised and $\bot$ is not $\mu$-realised.
\end{proof}

\section{Soundness \& Arithmetic}
In this section, we study the logic and arithmetic of $\mu$-realisability. We first observe that, in a certain sense, the Law of Excluded Middle is not $\mu$-realisable.

\begin{lemma}\label{Lemma:LEM}
%The Law of Excluded Middle is not randomly realised. 
There is $\phi$ such that $\forall{x}(\phi(x)\vee\neg\phi(x))$ is not $\mu$-realisable.
\end{lemma}
\begin{proof}
Let $\varphi(x)$ be the formula expressing the fact that the program $p_x(x)$ halts. Assume that $\forall x (\varphi(x)\lor \lnot \varphi(x))$ is randomly realised. Then, there is a program $p$ such that $p\Vdash_{\mu}\forall x (\varphi(x)\lor \lnot \varphi(x))$. Therefore, $p$ computes the halting problem for a set of oracles of measure ${>0}$. But this directly contradicts Sacks' theorem \cite[Corollary 8.12.2]{downey2010algorithmic}.
\end{proof}

We now show that $\mu$-realisability is preserved by the inference rules of first-order intuitionistic proof calculus.

First, we need to fix what it means for $\phi$ to be $\mu$-realizable when $x$ occurs freely in $\phi$: This is defined to mean the same as the $\mu$-realisability of $\forall{x}\phi$.

\begin{definition}[Intuitionistic Calculus]
Inference rules are:
\begin{align*}
\mathrm{MP}:& \text{ from } \phi \text{ and } \phi \to \psi \text{ infer }\psi\\
\forall-\mathrm{GEN}: & \text{ from } \psi \to \phi \text{ infer }\psi \to (\forall x \ \phi )\text{, if } x \text{ is not free in }\psi.\\
\exists-\mathrm{GEN}: &\text{ from } \phi \to \psi \text{ infer }(\exists x \ \phi ) \to \psi\text{, if }x \text{ is not free in }\psi.
\end{align*}
The axioms are
\begin{align*}
\mathrm{THEN}-1:& \phi \to (\chi \to \phi )\\
\mathrm{THEN}-2:& (\phi \to (\chi \to \psi )) \to ((\phi \to \chi ) \to (\phi \to \psi )) \\
\mathrm{AND}-1: & \phi \land \chi \to \phi \\
\mathrm{AND}-2: & \phi \land \chi \to \chi\\
\mathrm{AND}-3: & \phi \to (\chi \to (\phi \land \chi ))\\
\mathrm{OR}-1: &\phi \to \phi \lor \chi\\
\mathrm{OR}-2: & \chi \to \phi \lor \chi\\
\mathrm{OR}-3: &(\phi \to \psi ) \to ((\chi \to \psi ) \to ((\phi \lor \chi) \to \psi ))\\
\mathrm{FALSE}: & \bot \to \phi\\
\mathrm{PRED}-1: &(\forall x \ \phi (x)) \to \phi (t) \text{, if the term $t$ is free for substitution} \\ & \text{for the variable $x$ in $\phi$}\\
\mathrm{PRED}-2: & \phi (t) \to (\exists x \ \phi (x)) \text{, with the same restriction as for $\mathrm{PRED}-1$.}
\end{align*}
\end{definition}

\begin{proof}[Theorem \ref{Theorem:Soundness}]
We show that (i) all instantiations of the axioms of intuitionistic first-order calculus are $\mu$-realisable and (ii) the set of $\mu$-realizable statements is closed under modus ponens, $\forall$-GEN and $\exists$-GEN.

We start with (i).

%I write this here so we can check it. I propose to eventually delete the respective parts and replaces them with ```The $\mu$-realizability of the instantiations of the axioms ... are trivial.''

THEN-$1$: A $\mu$-realiser for an instance of $\phi\rightarrow(\chi\rightarrow\phi)$ needs to turn any given $\mu$-realiser $r$ for $\phi$ into one for $\chi\rightarrow\phi$. The $\mu$-realiser for $\chi\rightarrow\phi$ works by simply returning $r$ for any input.

THEN-$2$: Here, we are given a $\mu$-realiser $r$ for $\phi\rightarrow(\chi\rightarrow\psi)$ and our goal is to turn any $\mu$-realiser $p$ for $(\phi\rightarrow\chi)$ into a $\mu$-realiser $q$ for $\phi\rightarrow\psi$. Given $r$ and $p$, $q$ works as follows: Given a $\mu$-realiser $s$ for $\phi$, first use $r$ to compute from $s$ a $\mu$-realiser $t$ for $\chi\rightarrow\psi$ with positive probability; moreover, use $p$ to compute from $s$ a $\mu$-realiser for $u$ $\chi$ with positive probability. Then apply $t$ to $u$. %I AM UNSURE. SHOULD BE DOUBLE-CHECKED!

AND-$1$ works by projecting the $\mu$-realiser for $\phi\wedge\chi$ to the first component, AND-$2$ by projecting to the second component. %trivial

AND-$3$: We need to turn any $\mu$-realiser $p$ for $\phi$ into a $\mu$-realiser $q$ for $\chi\rightarrow(\phi\wedge\chi))$ with positive probability. Let $p$ be given. Also, let a $\mu$-realiser $r$ for $\chi$ be given. 
%THIS VERSION WORKS UNIFORMLY IN THE INSTANTIATIONS: 
Now, $q$ works as follows: For a given oracle $x$, let $x=x_{0}\oplus x_{1}$, where, for real numbers $a$ and $b$, $a\oplus b$ denotes the join of $a$ and $b$, i.e., $2i\in a\oplus b$ iff $i\in a$ and $2i+1\in a\oplus b$ iff $i\in b$. Now $(q^{x}(0),x)$ runs $(p^{x_{0}}(0),x_{0})$ while $(q^{x}(1)(0),x)$ runs $(r^{x_{1}},x_{1})$. 
%if one does not require uniformity, one can also note that $p$ and $r$ both work with probability $>0.9$ in some rational interval, so that combining them with the linear bijections mapping [0,1] into that interval makes them work on a set of measure 0.9; hence, there is a set of measure $>0.8$ on which both of them work.

OR-$1$ works by, given a $\mu$-realiser $r$ for $\phi$, sending $0$ to $0$ and $1$ to $r$, OR-$2$ by sending $0$ to $1$ and $1$ to $r$.

OR-$3$: We need to turn any $\mu$-realiser $p$ for $\phi\rightarrow\psi$ into one for $((\chi\rightarrow\psi)\rightarrow((\phi\vee\chi)\rightarrow\psi))$ with positive probability. Let $q$ be a $\mu$-realiser for $\chi\rightarrow\psi$, and let $r$ be a $\mu$-realiser for $\phi\vee\chi$. Now, the sets $S_{0}$, $S_{1}$ of oracles relative to which $r$ realizes $\phi$ or $\chi$, respectively, are measurable, and as their union has positive measure, at least one of the sets $S_{0}$ and $S_{1}$ has positive measure. %Assume without loss of generality that this is $S_{0}$. THAT DOES NOT HELP, WE WOULD NEED TO DECIDE WHICH OF THE TWO A CERTAIN r DOES? OR MAYBE NOT; we can just decide randomly and still retain positive probability?
Thus, for a positive measure set $S$ of oracles $u$, $r^{u}(0)$ will terminate with output $i\in\{0,1\}$ such that $S_{i}$ has positive measure, so that $(r^{u}(1),u)$ will be an $O$-realiser of $\chi$ (if $i=0$) or $\psi$ (if $i=1$), respectively. 
Let us denote by $r(1)$ the program that, on oracle $u$, runs the program with index $r^{u}(1)$ in the oracle $u$. With positive probability, $r(1)$ will be an $O$-realiser of $\phi$ (if $i=0$) or $\psi$ (if $i=1$).
Now we proceeds as follows: Given $u$, first compute $r^{u}(0)$. If this is $0$, apply $p$ to $r(1)$. If it is $1$, apply $q$ to $r(1)$. With positive probability, it then happens that $p$ is applied to a $\mu$-realiser of $\phi$ or that $q$ is applied to a $\mu$-realiser of $\chi$. In both cases, we obtain a $\mu$-realiser of $\psi$. 
Thus, we obtain a $\mu$-realiser of $\psi$ with positive probability, as desired.

FALSE is $\mu$-realized by any program, as $\bot$ does not have $\mu$-realisers.

PRED-$1$: Here, $t$ will just be a natural number. Let $r$ be a $\mu$-realiser for $\forall{x}\phi(x)$. Let an oracle $u$ be given, and suppose that $r$ works for $u$ (i.e., $(r,u)\Vdash_{O}\forall{x}\phi(x)$), which happens for all $u$ from a set of positive measure. For each such $u$, $(r^{u}(t),u)$ will be an $O$-realiser for $\phi(t)$ by definition. Thus, the program $r(t)$ that, for given $u$, runs the program with index $r^{u}(t)$ in the oracle $u$ is a $\mu$-realiser for $\phi(t)$.

PRED-$2$: Let $r$ be a $\mu$-realiser for $\phi(t)$. Then a $\mu$-realiser $p$ for $\exists{x}\phi(x)$ works by letting $p^{u}(1)$ output $t$ and letting $p^{u}(0)$ output $r$ for every $u$.

Now for (ii).

(1) (MP) If $\phi$ and $\phi\rightarrow\psi$ are $\mu$-realizable, then so is $\psi$.

Suppose that $p$ $\mu$-realizes $\phi$ and that $q$ $\phi$-realizes $\phi\rightarrow\psi$. Pick a real number $x$ such that $(q,x)$ realizes $\phi\rightarrow\psi$ and run $q^{x}(p)$. By definition, the output is a $\mu$-realiser for $\psi$.
%We need to compute a $\mu$-realizer for $\psi$. Consider the program $r$ that, given the input $x\oplus y$, computes $q^{x}(p)$ and executes the program whose index is the output of that calculation in the oracle $y$. With positive probability, $q^{x}(p)$ is the index of a $\mu$-realizer of $\psi$ and with positive probability, this $\mu$-realizer works relative to $y$. Thus, the set of $x\oplus y$ for which $(q^{x}(p),y)\Vdash_{O}\psi$ has positive measure and $r$ is a $\mu$-realizer for $\psi$. [DOUBLE-CHECK THIS! I FIND THIS STRANGE!]

%[My version:
%Suppose that $p$ $\mu$-realizes $\phi$ and that $q$ $\phi$-realizes %$\phi\rightarrow\psi$. Then there is $u\in 2^{\omega}$ such that %$(q,u)\Vdash_{\mathrm{O}}\phi\rightarrow \psi$. But then by definition $q^u(p)$ %is a $\mu$-realiser of $\psi$ as desired.
%]

(2) $\forall$-GEN

Let $p$ be a $\mu$-realiser for $\psi\rightarrow\phi$, and let $q$ $\mu$-realize $\psi\rightarrow(\forall{x}\phi)$, where $x$ does not occur freely in $\psi$ but (possibly) in $\phi$. If $x$ does not occur freely in $\psi$, the claim is trivial since then $\forall{x}\phi$ is $\mu$-realizable if and only if $\phi$ is. We are given $n\in\omega$, our goal is to compute a realiser for $\phi(n)$. Pick some oracle $y$ such that $(p,y)$ realizes $\psi\rightarrow\phi$. Note that this means that $(p,y)$ computes a realiser for $\psi\rightarrow\phi(n)$ from any given $n\in\omega$. Now run this realiser in the input $n$; by definition, the output will be a $\mu$-realiser for $\phi(n)$, as desired.
%where do we use the assumption that x is not free in $\psi$?

(3) $\exists$-GEN

Let $p$ be a $\mu$-realiser for $\phi\rightarrow\psi$ and let $q$ be a $\mu$-realiser for $(\exists{x}\phi)\rightarrow\psi$, where $x$ is not free in $\psi$. Pick oracles $y$ and $z$ such that $(q,y)$ realizes $(\exists{x}\phi)\rightarrow\psi$ and $(p,z)$ realizes $\phi\rightarrow\psi$. Thus, $(q^{y}(0),y)$ realizes $\phi(n)$, where $n$ is the output of $q^{y}(1)$. Recall that $p$ is a $\mu$-realiser for $\forall{x}(\phi\rightarrow\psi)$. Thus, $p^{z}(n)$ is the index of a program $r$ that turns $\mu$-realisers for $\phi(n)$ into $\mu$-realisers for $\psi$. Consequently, running $p^{z}(n)$ on the input $q(1)$ yields a $\mu$-realiser for $\psi$.
%THIS SHOULD BE DOUBLE-CHECKED.
\end{proof}

\begin{theorem}[Soundness]
\label{Theorem:Soundness}
The set of $\mu$-realizable statements is closed under the rules of intuitionistic first-order calculus.
\end{theorem}

%\begin{definition}[Intuitionistic Calculus]
%Inference rules are:
%\begin{align*}
%\mathrm{MP}:& \text{ from } \phi \text{ and } \phi \to \psi \text{ infer %}\psi\\
%\forall-\mathrm{GEN}: & \text{ from } \psi \to \phi \text{ infer }\psi \to %(\forall x \ \phi )\text{, if } x \text{ is not free in }\psi.\\
%\exists-\mathrm{GEN}: &\text{ from } \phi \to \psi \text{ infer }(\exists x \ %\phi ) \to \psi\text{, if }x \text{ is not free in }\psi.
%\end{align*}
%The axioms are
%\begin{align*}
%\mathrm{THEN}-1:& \phi \to (\chi \to \phi )\\
%\mathrm{THEN}-2:& (\phi \to (\chi \to \psi )) \to ((\phi \to \chi ) \to (\phi %\to \psi )) \\
%\mathrm{AND}-1: & \phi \land \chi \to \phi \\
%\mathrm{AND}-2: & \phi \land \chi \to \chi\\
%\mathrm{AND}-3: & \phi \to (\chi \to (\phi \land \chi ))\\
%\mathrm{OR}-1: &\phi \to \phi \lor \chi\\
%\mathrm{OR}-2: & \chi \to \phi \lor \chi\\
%\mathrm{OR}-3: &(\phi \to \psi ) \to ((\chi \to \psi ) \to ((\phi \lor \chi) %\to \psi ))\\
%\mathrm{FALSE}: & \bot \to \phi\\
%\mathrm{PRED}-1: &(\forall x \ \phi (x)) \to \phi (t) \text{, if the term $t$ %is free for substitution for the variable $x$ in $\phi$}\\
%\mathrm{PRED}-2: & \phi (t) \to (\exists x \ \phi (x)) \text{, with the same %restriction as for $\mathrm{PRED}-1$.}
%\end{align*}
%\end{definition}

It is a classical result that the axioms of Heyting Aritmetic are realisable, see \cite[Theorem 1]{Nelson}. We show that only a fragment of $\mathrm{HA}$ is $\mu$-realisability. Let $\mathrm{HA}^-$ denote the axioms of Peano arithmetic without the induction schema.
%\begin{align*}
%&\forall x \ (0 \neq  \mathrm{S} ( x ))\\
%&\forall x, y \ (\mathrm{S} ( x ) =  \mathrm{S} ( y ) \Rightarrow x = y)\\
%&\forall x \ (x  + 0 = x )\\
%&\forall x, y \ (x + \mathrm{S} ( y ) =  \mathrm{S} ( x + y ))\\
%&\forall x \ (x \cdot 0 = 0)\\
%&\forall x, y \ (x \cdot  \mathrm{S} ( y ) = x \cdot y + x )
%\end{align*}
%As usual the induction schema is the recursively enumerable set of axioms 
%$$\forall \bar{y} ((\varphi(0,\bar{y}) \land \forall x ( \varphi(x,\bar{y})\Rightarrow\varphi(S(x),\bar{y}))) \Rightarrow \forall x \varphi(x,\bar{y}))$$
%where $\varphi(x,\bar{y})$ is a formula in the language of arithmetic.
As usual, \emph{Heyting arithmetic $\mathrm{HA}$} is the theory obtained from adding the induction schema to $\mathrm{HA}^{-}$. We say that a set of formulas $\Gamma$ is $\mu$-realised if $\varphi$ is $\mu$-realised for all $\varphi\in \Gamma$.

Since all the axioms except for the induction schema are universal $\Pi_1$ statements, it follows by Theorem \ref{Lem:RecClass} that the axioms of $\mathrm{HA}^-$ are all $\mu$-realised.

\begin{theorem}\label{Theorem Heyting Aritmetic}
The set $\mathrm{HA}^-$ is $\mu$-realised.    
\end{theorem}
%\begin{proof}
    %Follows from the fact that all the axioms of $\mathrm{HA}^{-}$ are %universal $\Pi_1$ statements and from Theorem \ref{Lem:RecClass}.
%\end{proof}

Contrary to the classical case the induction schema fails for $\mu$-realisability.  

\begin{theorem}\label{Theo:IndFail}
The induction schema is not $\mu$-realised.  
\end{theorem}
\begin{proof}
Let $\varphi(x)$ be the formula expressing the fact that \vir{Every program with code $i<x$ halts or does not halt}. By the proof of Lemma \ref{Lemma:LEM}, $\varphi$ is not $\mu$-realisable.
%Note that, a $\mu$-realiser of $\forall{n}\varphi(n)$ would be a program that for a positive measure computes the halting problem, but this contradicts Sacks' theorem \cite[Corollary 8.12.2]{downey2010algorithmic}.

On the other hand, a $\mu$-realiser $p(n)$ for $\varphi(n)$ is given by a program that does the following: for every $i<n$, $p$ returns a program that if the $i$th element of the oracle is $1$ returns $1$ on input $0$ and any number on input $1$. While if the $i$th element of the oracle is $0$ the program returns $0$ on input $0$ and on input $1$ starts building a realiser of \vir{the program $i$ halts} using the algorithm in Lemma \ref{Cor:CompClassicalRealisability}; if it finds one, it runs the algorithm in Theorem \ref{Lem:RecClass} to compute the desired $\mu$-realiser.

It is not hard to see that the algorithm  works with probability $\frac{1}{2^n}$. Thus, to realize the implication $\varphi(n)\rightarrow \varphi(n+1)$, we can ignore the $\mu$-realiser for $\varphi(n)$ and just output $p(n)$. So the premise of the instance of the induction schema is $\mu$-realised, while the conclusion is not.
%not $\mu$-realised. 
\end{proof}

%\begin{remark}
Note that the proof of Theorem \ref{Theo:IndFail} heavily relies on the fact that the definition of $\mu$-realisability does not require any relationship between the measures of the set of oracles realising the antecedent of an implication and the set of oracles realising the consequent. We think that a modification of this definition could lead to a notion of probabilistic realisability that realises the induction schema.  
%\end{remark}

Even though the axiom schema of induction is not $\mu$-realisable, one can prove that all $\Delta_0$-instances of the schema are realisable. Indeed, by Theorem \ref{Lem:RecClass} and the fact that if $\varphi$ is a $\Delta_0$ formula then $\forall x \varphi(x,\bar{y})$ is a universal $\Pi_1$ formula, we have the following:

\begin{corollary}
The set $\mathrm{HA}^{-}$ together with the induction schema restricted to $\Delta_0$ formulas is $\mu$-realisable.
\end{corollary}

\section{Big Realisability}
\label{Sec:BigReal}

In this section, we will consider other natural definitions of realisability arising from notions of \textit{big sets of oracles} on the real numbers. More specifically, we will consider ``almost sure realisability,'' ``comeagre realisability,'' ``interval-free realisability,'' and ``positive measure realisability.'' It will turn out, however, that the first three are equivalent to standard realisability, while the final one coincides with truth. We begin with the following general definition.

% We obtain the concept of an ``almost sure realizer'' or a ``$1$-realizer'' by demanding the measure of the set of oracles to be $1$ in the definition of $\mu$-realisability. We say that $\phi$ is ``almost surely realizable'' or ``$1$-realizable'' if and only if it has a $1$-realiser. [NO WE NEED A NEW RECURSIVE DEFINITION IF YOU DO THIS YOU OBTAIN SOMETHING DIFFERENT. COMMENT ON THIS] 

\begin{definition}
    Let $\mathcal{F}$ be a family of subsets of Cantor space $2^\omega$. We then define $\mathcal{F}$-realisability recursively as follows:
    \begin{enumerate} 
        \item $p \Vdash_{\mathcal{F}} \bot$ never,
        \item $p \Vdash_{\mathcal{F}} n=m $ if and only if $n=m$,
        \item $p \Vdash_{\mathcal{F}} \psi_0 \wedge \psi_1$ if and only if $p(i) \Vdash_{\mathcal{F}} \psi_i$ for $i < 2$,
        \item $p \Vdash_{\mathcal{F}} \psi_0 \vee \psi_1$ if and only if there is some $O \in \mathcal{F}$ and some $i < 2$ such that for every $u \in O$, we have $p^u(0) = i$ and $p^u(1) \Vdash_{\mathcal{F}} \psi_i$,
        \item $p \Vdash_{\mathcal{F}} \phi \rightarrow \psi$ if and only if there is a set $O \in \mathcal{F}$, such that for every $u \in O$ and $s \Vdash_{\mathcal{F}} \phi$, we have $p^u(s) \Vdash_{\mathcal{F}} \psi$,
        \item $p \Vdash_{\mathcal{F}} \exists x \phi$ if and only if there is some $O \in \mathcal{F}$, such that there is some $n$ for all $u \in O$ such that $p^u(0) = n$ and $p^u(1) \Vdash \phi(n)$,
        \item $p \Vdash_{\mathcal{F}} \forall x \phi$ if and only if there is a set $O \in \mathcal{F}$, such that for every $u \in O$ and $n\in \mathbb{N}$ we have $p^u(n) \Vdash_{\mathcal{F}} \phi(n)$.
    \end{enumerate}
\end{definition}

From this definition, we derive the following notions of realisability: Let $\mathcal{F}_\mathrm{if}$ be the family of \emph{co-interval-free} subsets of the Cantor space, i.e. $X \in \mathcal{F}_\mathrm{cif}$ if and only if $X \in 2^\omega$ and there is no open interval $I$ such that $I \subseteq 2^\omega \setminus X$, and $\Vdash_\mathrm{cif}$ denotes $\mathcal{F}_\mathrm{cif}$-realisability. 
Let $\mathcal{C}$ be the family of comeagre subsets of the Cantor space, then let $\Vdash_\mathcal{C}$ denote $\mathcal{C}$-realisability. 
Let $\mathcal{F}_{=1}$ be the family of subsets of the Cantor space that are of measure $1$, and let $\Vdash_{=1}$ denote $\mathcal{F}_{=1}$-realisability.
Let $\mathcal{F}_{>0}$ be the family of subsets of the Cantor space of positive measure, and $\Vdash_{>0}$ denotes $\mathcal{F}_{>0}$-realisability.
As before, we will write $\Vdash_\mathcal{F}\varphi$ if and only if there is some realiser $p$ such that $p \Vdash_\mathcal{F} \phi$.

In what follows we will make use of the \emph{bounded exhaustive search with $p(n)$}, i.e. the following procedure. Given a program $p$ (and possibly some input $n$), do the following successively for all $k \in \omega$.
% \begin{itemize}
    %\item[(i)] 
    Enumerate all $0$-$1$-strings of length $k$. For each of these strings $s$, do the following:
    % \item[(ii)] 
    Run $p^{s}(n)$ for $k$ many steps. %If the computation demands more than the first $k$ bits of the oracle, continue with the next $s$ (if there is one, otherwise with $(k+1)$). 
If the computation does not halt within that time (which implies in particular that at most the first $k$ many bits of the oracle were requested), continue with the next $s$ (if there is one, otherwise with $(k+1)$). If the computation halts with output $x$ within that time, then the search terminates with output $x$.
% \end{itemize}

The crucial property of this procedure, which is also contained in the proof idea of Sacks' theorem \cite[Corollary 8.12.2]{downey2010algorithmic}, is the following:

\begin{lemma}{\label{bounded exhaustive search}}
    Let $G \subseteq \omega$, $n\in\omega$ and let $p$ be a program. %such that $p^{u}(n)$ halts for all $u\subseteq\omega$. 
    Suppose that there is a set $S \subseteq 2^\omega$ such that $2^\omega \setminus S$ is interval-free
    %is nowhere dense %of measure $1$ 
    and $p^{u}(n)$ terminates for all $u\in S$ with output $k\in G$. Then the bounded exhaustive search with $p(n)$ will terminate with output $k\in G$.
\end{lemma}
\begin{proof}
    %As in the proof of Sacks' theorem, if the bounded exhaustive search would not terminate at all, there would be finite $0$-$1$-strings $s$ of arbitrary length such that $P^{s}(n)$ does not terminate. This means that the set of those sequences is an infinite, finitely branching binary tree $t$. By K\"onig's lemma, $t$ has an infinite branch $b$. But then, $P^{b}(n)$ would not terminate, which contradicts the assumption on $P$ (cf. Kaye, LOGIC BOOK, maybe just refer to this).
    
    Note that for every $n$ and $u\in S$ we have that $p^u(n)$ terminates with output in $G$. So there is a finite initial segment $s$ of $u$ such that $p^s(n)$ terminates with output $p^u(n)$. So, the bounded exhaustive search will halt. 

    Now, note that if the search halts on the string $s$ with output $k\in\omega$, but $k\notin G$, then $p^{x}(n)\downarrow k$ for all $u\in \mathrm{N}_s$. But then, $\mathrm{N}_s \subseteq 2^\omega\setminus S$ which contradicts the fact that $2^\omega\setminus S$ is interval free.

    %the real numbers $y$ for which $P^{y}(n)$ terminates with an element of $G$ as their output all belong to the complement of $\mathrm{N}_{s}$  which contradicts the fact that $2^\omega\setminus S$ is nowhere dense.
\end{proof}

\begin{lemma}
    \label{Lemma: meagre and measure 0 contain no intervals}
    Let $X \subseteq 2^\omega$ be a subset of Cantor space. If $\mu(X) = 0$ or $X$ is meagre, then $X$ is interval-free.
\end{lemma}
\begin{proof}
    The first statement follows trivially from the fact that every non-empty open interval has positive measure. For the second statement, recall that meagre sets have empty interiour by the Baire Category Theorem (cf. \cite[Theorem 0.11]{kanamori}) and therefore contain no intervals.
\end{proof}

\begin{theorem}
    \label{Theorem: F-realisability and realisability are the same}
    Let $\mathcal{F}$ be a family of subsets of Cantor space such that every $X \in \mathcal{F}$ is co-interval-free. There are programs $P_\mathcal{F}$ and $P_\mathcal{F}^{-1}$ such that the following holds for all statements $\phi$: (i) if $p \Vdash \phi$, then $P_\mathcal{F}(p,\varphi) \Vdash_\mathcal{F} \phi$, (ii) if $p \Vdash_\mathcal{F} \phi$, then $P_\mathcal{F}^{-1}(p,\varphi) \Vdash \phi$.
    Consequently, $\phi$ is realisable if and only if it is $\mathcal{F}$-realisable, and $\Vdash$, $\Vdash_\mathrm{cif}$, $\Vdash_\mathcal{C}$, and $\Vdash_{=1}$ coincide.
    %${\Vdash} = {\Vdash_\mathrm{cif}} = {\Vdash_\mathcal{C}} = {\Vdash_{=1}}$.
\end{theorem}
\begin{proof}
    We show both statements by simultaneous induction on the complexity of $\phi$ and simultaneously define $P_\mathcal{F}$ and $P_\mathcal{F}^{-1}$ by recursion on $\phi$.
    
    \bigskip
    
    (1) $\phi$ is $t_0 = t_1$ or $t_0 \neq t_1$. In this case, $\mathcal{F}$-realisers and realisers are the same, so the statement is trivial: $P_\mathcal{F}$ and $P_\mathcal{F}^{-1}$ just return the first component.
    
    \bigskip
    
    (2) $\phi$ is $\psi_0 \wedge \psi_1$. 
    
    Let $r = (r_0,r_1)$ be a realiser for $\phi$ such that $r_i$ realises $\psi_i$ for $i<2$. By induction hypothesis, $P_\mathcal{F}(r_i,\psi_i)$ will return an $\mathcal{F}$-realiser for $\psi_i$. Hence, $P_\mathcal{F}(r,\phi)$ is the program that outputs $P_\mathcal{F}(r_i,\psi_i)$ on input $i$. We obtain $P_\mathcal{F}^{-1}$ in exactly the same way.
    
    \bigskip
    
    (3) $\phi$ is $\psi_0 \vee \psi_1$.
    
    Let $r$ be a realiser for $\phi$, i.e. $r(0)$ returns some $i < 2$ and $r(1) \Vdash \psi_i$. By induction hypothesis, we have that $P_\mathcal{F}(r(1),\psi_i) \Vdash_\mathcal{F} \psi_i$. Hence, $P_\mathcal{F}(r,\phi)$ is the program that returns $i$ on input $0$ and $P_\mathcal{F}(r(1),\psi_i)$ on input $1$.
    
    Conversely, let $r$ be an $\mathcal{F}$-realiser for $\phi$. Then there are some $i < 2$ and $O \in \mathcal{F}$ such that for all $u \in O$, $r^u(0) = i$ and $r^u(1) \Vdash_\mathcal{F} \psi_i$. Hence, let $P_{\mathcal{F}}^{-1}(r, \phi)$ be the program that executes a bounded exhaustive search with $r^u(0)$, which terminates by Lemma \ref{bounded exhaustive search} in some $i < 2$, and then returns $i$ on input 0, and $P_\mathcal{F}^{-1}(r, \psi_i)$ on input 1. Then $P_\mathcal{F}^{-1}(r) \Vdash \phi$.
    
    \bigskip
    
    (4) $\phi$ is $\psi_0 \rightarrow \psi_1$.
    
    Let $r \Vdash \phi$. Then $r$ is a program that, given a realiser $r_0 \Vdash \psi_0$, returns a realiser $r_1 \Vdash \psi_1$. Let $r_0' \Vdash_\mathcal{F} \psi_0$. By induction hypothesis, $P_\mathcal{F}^{-1}(r_0',\psi_0) \Vdash \psi_0$. Hence, $r(P_\mathcal{F}^{-1}(r_0',\psi_0)) \Vdash \psi_1$ and $P_\mathcal{F}(r(P_\mathcal{F}^{-1}(r_0',\psi_0)),\psi_1) \Vdash_\mathcal{F} \psi_1$. Therefore, let $P_\mathcal{F}(r,\phi)$ be the program that takes a realiser $r_0' \Vdash \psi_0$ as input and returns $P_\mathcal{F}(r(P_\mathcal{F}^{-1}(r_0',\psi_0)),\psi_1)$.
    
    The proof for the other direction is symmetric by exchanging the roles of $P_\mathcal{F}$ and $P_\mathcal{F}^{-1}$.
    
    \bigskip 
    
    (5) $\phi$ is $\exists x \psi(x)$. 
    
    Let $r \Vdash \exists x \psi(x)$. Then $r(0) = n$ and $r(1) \Vdash \psi(n)$. By induction hypothesis, it follows that $P_\mathcal{F}(r(1),\psi) \Vdash_\mathcal{F} \psi(n)$. So let $P_\mathcal{F}(r,\phi)$ be the program that output $n$ on input $0$, and $P_\mathcal{F}(r(1),\psi)$ on input $1$. Then, $P_\mathcal{F}(r,\phi) \Vdash_\mathcal{F} \phi$.
    
    Conversely, let $r \Vdash_\mathcal{F} \exists x \psi(x)$. Then there is some $O \in \mathcal{F}$ and $n \in \omega$ such that $r^u(0) = n$ and $p^u(1) \Vdash_\mathcal{F} \psi(n)$. By induction hypothesis, $P_\mathcal{F}^{-1}(p^u(1),\psi) \Vdash \psi(n)$. Define $P_\mathcal{F}^{-1}(r,\phi)$ to be the following program: First, start a bounded exhaustive search with $r(0)$. By Lemma \ref{bounded exhaustive search} this search must terminate with output $n$. Return $n$ on input $0$, and return $P_\mathcal{F}^{-1}(r^u(1),\psi)$ on input $1$. Then $P_\mathcal{F}^{-1}(r,\phi) \Vdash \exists x \psi(x)$.
    
    \bigskip 
    
    (6) $\phi$ is $\forall x \psi(x)$. 
    
    Let $r \Vdash \forall x \psi(x)$. Then $r(n) \Vdash \psi(n)$ for every $n \in \omega$. Let $P_\mathcal{F}(r,\phi)$ be the program that, given $n \in \omega$, returns $P_\mathcal{F}(r(n),\psi)$. With the induction hypothesis, it follows that $P_\mathcal{F}(r,\phi) \Vdash_\mathcal{F} \phi$.
    
    Conversely, let $r \Vdash_\mathcal{F} \exists x \psi(x)$. Then there is some $O \in \mathcal{F}$ such that for every $u \in O$ and $n \in \mathbb{N}$ we have that $r^u(n) \Vdash \psi(n)$. Define $P_\mathcal{F}^{-1}(r,\phi)$ to be the following program: Start a bounded exhaustive search with $r(n)$. By Lemma \ref{bounded exhaustive search}, this search will terminate with $r' \Vdash_\mathcal{F} \psi(n)$. Then return $P_\mathcal{F}^{-1}(r',\psi)$, which, by induction hypothesis, is a realiser of $\psi(n)$. Hence, $P_\mathcal{F}^{-1}(r,\phi) \Vdash \psi$.
\end{proof}

\begin{theorem}\label{Theo: non null and truth}
    Let $\phi$ be a formula. Then $\Vdash_{>0} \phi$ if and only if $\phi$ is true.
\end{theorem}
\begin{proof}
The proof is an induction on the complexity of $\varphi$. 

\bigskip

(1) If $\varphi$ is atomic the statement follows by the definitions. 

\bigskip

(2) Assume that $\varphi\equiv \psi_0 \land \psi_1$.\\

If $\phi$ is true then by inductive hypothesis there are $p$ and $q$ such that $p$ $\mathcal{F}_{>0}$-realises $\psi_0$ and $q$ $\mathcal{F}_{>0}$-realises $q$. Let $s$ be a sequence which starts with a code of $p$ followed by a marker and by a code for $q$ followed by a second marker. Then let $t$ be the program that on input $0$ returns the content of the oracle up to the first marker and on input $1$ returns the content of the oracle between the first and the second marker. Note that for all $u\in \mathrm{N}_s$, $r^u(0)\Vdash_{{>0}} \psi_0$ and $r^u(1)\Vdash_{>0}\psi_1$. So, $r\Vdash_{>0}\phi$ as desired.

On the other hand if $\varphi$ is $\mathcal{F}_{>0}$-realised then by definition both $\psi_0$ and $\psi_1$ are $\mathcal{F}_{>0}$-realised and the statement follows by the inductive hypothesis.  

\bigskip

(3) Assume that $\varphi\equiv \psi_0 \lor \psi_1$.\\

If $\phi$ is true then by inductive hypothesis there is $p$ such that $p^u(1)\Vdash_{>0}\psi_{p^u(0)}$ for every $u$ in some positive measure set $O$. Let $s$ be a sequence which starts with $p^u(0)$ followed by  a code for $p^u(1)$ followed by a marker. Then let $q$ be the program that on input $0$ returns the content of the first bit of the oracle and on input $1$ returns the content of the oracle from the second bit to the marker. Note that for all $u\in \mathrm{N}_s$, $q^u(1)\Vdash_{>0}\psi_{q^u(0)}$. So, $q\Vdash_{>0}\phi$ as desired.

On the other hand if $\varphi$ is $\mathcal{F}_{>0}$-realised then by definition at least one between $\psi_0$ and $\psi_1$ is $\mathcal{F}_{>0}$-realised and the statement follows by the inductive hypothesis. 

\bigskip

(4) Assume that $\varphi\equiv \psi_0 \rightarrow \psi_1$. \\

Assume that $\varphi$ is true. Then either $\psi_1$ is true or $\psi_0$ is false. If $\psi_0$ is false then by inductive hypothesis is not $\mathcal{F}_{>0}$-realised and therefore any natural number will $\mathcal{F}_{>0}$-realise $\phi$. If $\psi_1$ is true, then by inductive hypothesis is $\mathcal{F}_{>0}$-realised by some program $p$. Let $s$ be the sequence that starts with a code of $p$ followed by a marker. Let $q$ the program that for every $n$ and every oracle returns the content of the oracle up to the first occurrence of the marker. Then for all $u\in \mathrm{N}_s$ and for every $n$ we have that $q^u(n)\Vdash_{>0} \psi_1$. So, $q\Vdash_{>0}\phi$ as desired.

On the other hand if $\varphi$ is $\mathcal{F}_{>0}$-realised by some program $p$. If $\psi_0$ is true then it is $\mathcal{F}_{>0}$-realised by some program $q$. Then there is a non-null set $O$ such that for all $u\in O$ we have that $p^u(q)\Vdash_{>0}\psi_1$. But then by inductive hypothesis $\psi_1$ must be true.

\bigskip

(5) Assume that $\varphi\equiv \exists{x}\psi$. \\

Assume that $\varphi$ is true. Then for some $n\in\mathbb{N}$ we have that $\psi(n)$ is true. By inductive hypothesis there is $p$ which $\mathcal{F}_{>0}$-realises $\psi(n)$. Let $s$ be a sequence starting with a code for $n$ followed by a marker and then by the code of $p$ followed by a marker. Let $q$ be the program that on input $0$ returns the content of the oracle up to the first marker, and on input $1$ returns the content of the oracle between the first and second marker.  Then for all $u\in \mathrm{N}_s$ and for every $n$ we have that $q^u(1)\Vdash_{>0}\psi(q^u(0))$. So, $q\Vdash_{>0}\phi$ as desired.

On the other hand if $\varphi$ is $\mathcal{F}_{>0}$-realised by some program $p$. Then there is a non-null set $O$ such that for all $u\in O$ we have that $p^u(1)\Vdash_{>0}\psi(p^u(0))$. But then by inductive hypothesis $\psi_1$ must be true.

\bigskip

(6) Assume that $\varphi\equiv \forall{x}\psi$. 

Assume that $\varphi$ is true. Then for all $n\in\mathbb{N}$ we have that $\psi(n)$ is true. Without loss of generality we can assume that the main operator of $\psi$ is not a universal quantifier, the proof can be easily modified otherwise. Let $q$ be the program that ignores the oracle and depending on the main connective of $\psi$ does the following: 
\begin{itemize}
\item if $\psi$ is atomic $q$ is just the constant function $1$;
\item if $\psi$ is $\psi_0\land \psi_1$ then $q(n)$ is the program $r$ from the proof of case (2);
\item if $\psi$ is $\psi_0\lor \psi_1$ then $q(n)$ is the program $q$ from the proof of case (3);
\item if $\psi$ is $\psi_0\rightarrow \psi_1$ then $q(n)$ is the program $q$ from the proof of case (4);
\item if $\psi$ is $\exists{x}\psi_0$ then $q(n)$ is the program $q$ from the proof of case (5);
\end{itemize}
By inductive hypothesis and by (2), (3), (4), and (5) of this proof we have that for all $u\in 2^\omega$ and for every $n$ we have that $q^u(n)\Vdash_{>0}\psi(n)$. So, $q\Vdash_{>0}\phi$ as desired.

On the other hand if $\varphi$ is $\mathcal{F}_{>0}$-realised by some program $p$. Then there is a non-null set $O$ such that for all $u\in O$ we have that $p^u(n)\Vdash_{>0}\psi(n)$. But then by inductive hypothesis $\phi$ must be true.
\end{proof}

%The following corollary is an immediate consequence.

%\begin{corollary}
%    Every instance of the law of excluded middle is $\mathcal{F}_{>0}$-realised. Also, the halting problem is $\mathcal{F}_{>0}$-realised. Hence, $\Vdash_{>0}$ and $\Vdash$ are not the same.
%\end{corollary}

\bibliographystyle{splncs04}
\bibliography{REFERENCES}
\end{document}